\documentclass[12pt,a4paper,oneside,titlepage]{article}
\usepackage{verbatim} 
\usepackage{amsmath,amsfonts,amssymb,amsthm,amscd}
\usepackage{xcolor}
\theoremstyle{plain}
\newtheorem{Theorem}{Theorem}

\newtheorem{Lemma}[Theorem]{Lemma}
\newtheorem{Corollary}[Theorem]{Corollary}

\theoremstyle{definition}
\newtheorem{Definition}[Theorem]{Definition}

\theoremstyle{definition}

\newtheorem{Remark}{Remark}
\newtheorem{Open Problem}[Remark]{Open Problem}

\newtheorem{Example}{Example}
\newcommand{\bC}{\mathbb C}
\newcommand{\bN}{\mathbb N}
\newcommand{\bP}{\mathbb P}

\newcommand{\bQ}{\mathbb Q}
\newcommand{\bZ}{\mathbb Z}

\newcommand{\cP}{\mathcal{P}}

\title{Algebraic independence of infinite series
\footnotetext{AMS Class: 11J85, 11J72.}
 }
\author{Jaroslav Han\v{c}l, Simon Kristensen, and Mathias L. Laursen}

\begin{document}
\maketitle
\date{}

\begin{abstract}
We give conditions on a finite set of series of rational numbers to ensure that they are algebraically independent. Specialising our results to polynomials of lower degree, we also obtain new results on irrationality and $\mathbb{Q}$-linear independence of such series.  
\end{abstract}

\section{Introduction}

The study of irrationality and transcendence of numbers date back to antiquity, where an unfortunate Pythagorean lost his life in proving the irrationality of $\sqrt{2}$. Much later, in the late nineteenth century, the existence of transcendental numbers was proven, explicit examples were provided, and some famous numbers such as $\pi$ and $e$ were shown to be transcendental. Later on, concepts such as linear independence over $\mathbb{Q}$, over $\overline{\mathbb{Q}}$ and algebraic independence of numbers have been important. In the first two cases, one thinks of the set of real numbers $\mathbb{R}$ as a vector space defined over the rationals, $\mathbb{Q}$, or the algebraic numbers, $\overline{\mathbb{Q}}$, and investigate whether a set of numbers is linearly independent over the base field as vectors. For algebraic independence, one investigates whether a set of $K$ real numbers has the property that no non-zero integer polynomial in $K$ variables vanishes at this set. In this case, the numbers are said to be algebraically independent. Of course, if one considers only polynomials of degree $1$, the notion of linear independence over $\mathbb{Q}$ is rediscovered. If furthermore $K = 1$, the question of irrationality is rediscovered.

The present paper is concerned with algebraic independence of a finite set of real numbers given in terms of series of rationals. More precisely, we will derive conditions on such a set of series which ensure that the numbers are algebraically independent. An important result in this respect is the 1975 result of Erd\H{o}s \cite{Erdos}, where it is shown that if $\{a_n\}_{n=1}^\infty$ is an increasing sequence of positive integers with $a_n > n^{1+\epsilon}$ for some $\epsilon > 0$ and 
\[
\limsup_{n \rightarrow \infty} a_n^{1/2^n} = \infty,
\]
then the number $\sum_{n=1}^\infty \frac{1}{a_n}$ is irrational. The result has been extended numerous times by various authors to encompass a wider range of series \cite{hanLinIndep} as well as series with more general denominators \cite{hns, ak, l}.

Our objective here is to obtain criteria ensuring algebraic independence of a set of such series. We will allow for numerators as well as denominators, although our series will remain defined over the rationals. One novelty of our approach is that we are able to remove the restriction of positivity in Erd\H{o}s' result at some expense. The conditions obtained will look somewhat cumbersome and technical. In order to demonstrate their applicability, we will provide a large number of examples of concrete series, which we show to be irrational, linearly independent over $\mathbb{Q}$ or even algebraically independent.

\section{Main results}
\newcommand{\aknbarinit}[2]{2^{#2^{-3} #1^{#2}}}
\newcommand{\aknbarafter}[2]{2^{#1N^{-2} (#2)^{N}}}

We will now state our main theorem. We will be considering series of the form
\[
\alpha_k = \sum_{n=1}^\infty \frac{b_{k,n}}{a_{k,n}},
\]
where for $k=1,2,\dots, K$, $\{a_{k,n}\}_{n=1}^\infty$ and $\{b_{k,n}\}_{n=1}^\infty$ are sequences of non-zero integers. Our results depend heavily on the joint divisibility properties
of the denominators. In our main result, we express the required properties in terms of the $p$-adic valuation. As usual, for a prime $p$ and an integer $n$, we will let $\nu_p(n)$ denote the $p$-adic valuation of $n$, i.e., $\nu_p(0)=\infty$ and, for $n\ne 0$, $\nu_p(n)$ is the unique non-negative integer such that $p^{\nu_p(n)} \vert n$ but $p^{\nu_p(n) + 1} \nmid n$.
With this in mind, we now state our main result on the algebraic independence of a finite set of series of the above form.
In some of the theorems, we give conditions to ensure that non-zero polynomials $P\in\bZ[x_1,\ldots, x_K]$ of a degree bounded by a fixed integer $d$ will all satisfy $P(\alpha_1,\ldots,\alpha_K)\ne 0$.
We here take the degree of a polynomial to be the maximum degree of its monomial terms, and we define $\deg(x_1^{i_1}\cdots x_K^{i_K}) = i_1+\cdots+i_K$.

When only two series are considered, so that $K = 2$, we are able to obtain much weaker criteria at the cost of arriving at a somewhat weaker conclusion. To fix ideas, we define a weaker notion here. For a polynomial $P(x_1, x_2)$, let $\tilde{P}(x_1, x_2, x_3)$ be the projective version of $P$, i.e., the homogeneous polynomial obtained from $P$ by multiplying each monomial with $x_3$ sufficiently many times that the resulting polynomial is homogeneous of the same degree as the original one. This polynomial defines a plane, projective curve, $\mathcal{C}_P \in \mathbb{P}^2$. If this curve is smooth, its (geometric) genus $g(\mathcal{C}_P)$  is given by the genus--degree formula,
\[
g(\mathcal{C}_P) = \frac{(\deg P-1)(\deg P-2)}{2}.
\]
If the curve is singular, the genus decreases. An ordinary singularity of multiplicity $r$ decreases the genus by $r(r-1)/2$. Non-ordinary singularities need to be examined individually. We will say that two numbers $\alpha_1$ and $\alpha_2$ are \emph{non-degenerately independent of order $d$} if no integer polynomial $P$ in two variables of degree at most $d$ such that $\deg P\le 3$ or $g(\mathcal{C}_P) \ge 2$ vanishes at the point $(\alpha_1, \alpha_2)$. If two numbers $\alpha_1$ and $\alpha_2$ are non-degenerately independent of order $d$ for any $d$, we will say that  $\alpha_1$ and $\alpha_2$ are \emph{non-degenerately algebraically independent}.

 For polynomials $P$ such that the resulting plane, projective curve is smooth, the genus condition is satisfied as soon as $\deg P > 3$. For higher degrees, we are removing certain potential algebraic dependencies with this definition, arising from the decrease in genus from the singularities. One should stress that most polynomials give rise to a smooth curve, and that if the degree is large, most non-smooth curves will have genus $\ge 2$. In work in progress \cite{kp}, F. Pazuki and the second named author obtain a counting estimate for polynomials of fixed degree $d$ resulting in a curve of fixed genus $g$.

\begin{Theorem}
	\label{thm:general}
	Let $K$ and $d$ be positive integers, let $\varepsilon$ and $\kappa$ be positive real numbers with $\kappa <1$, and let $p$ be a prime number.
	For $k=1,\ldots, K$, let $\{a_{k,n}\}_{n=1}^\infty$ and $\{b_{k,n}\}_{n=1}^\infty$ be sequences of non-zero integers with $p\nmid \gcd(a_{k,n},b_{k,n})$ such that for each sufficiently large $N\in\bN$,
	\begin{equation}\label{hkl1}
		\nu_p(a_{k,n}) = \max_{1\le m\le N} \nu_p(a_{k,m})
	\end{equation}
	for exactly one $n\le N$ and
	\begin{equation}
		\label{hkl2}
		\lim_{N\to\infty} \Big(\max_{1\le n\le N}\nu_p(a_{k,n}) - d \max_{1\le n\le N}\nu_p(a_{k-1,n})\Big)=\infty,
	\end{equation}
	writing $a_{0,n}=1$ for all $n$.
	Suppose there is a sequence $\{a_n\}_{n=1}^\infty$ of integers such that for every $k=1,\dots ,K$ and each sufficiently large $n\in\bN$,
	\begin{equation}
		\label{hkl4}
		n^{1+\varepsilon}\leq a_n\le a_{n+1},
	\end{equation}  
	\begin{equation}
		\label{hkl5}
		a_{n} 2^{-(\log_2 a_{n})^\kappa} \le |a_{k,n}| \leq \max \big\{a_{n} 2^{(\log_2 a_{n})^\kappa}, 2^{n^{-3}(Kd+1)^n}\big\},
	\end{equation}  
	\begin{equation}
		\label{hkl6}
		|b_{k,n}| \leq 2^{(\log_2 a_{n})^\kappa},
	\end{equation}  
	and
	\begin{equation}
		\label{hkl7}
		\limsup_{n\rightarrow\infty} a_{n}^{\frac{1}{(Kd+1)^n}} =
		\infty .
	\end{equation}
	For $k=1,\dots ,K$, set	$\alpha_k= \sum_{n=1}^\infty \frac {b_{k,n}}{a_{k,n}}$. Then $(\alpha_1,\dots ,\alpha_K)$ is not the root of any non-zero polynomial of $K$ variables with integer coefficients and degree less than or equal to $d$.

	If $K=2$, this is also true when $d$ is replaced by 3 in equation \eqref{hkl2} and the non-vanishing conclusion is replaced by non-degenerately independent of order $d$.
\end{Theorem}

It is worth noting that in the last case of the theorem, if one replaces $3$ with any higher number $d'$, say, \eqref{hkl2} automatically rules out the existence of polynomials of degree $\le d'$ vanishing at the point $(\alpha_1, \alpha_2)$.  This cuts down on the number of degenerate polynomials potentially obstructing full independence.

The conditions of Theorem \ref{thm:general} may seem unwieldy to check, and their origin may not be terribly clear. Let us take a moment to digest them. Condition \eqref{hkl4}, \eqref{hkl5}, \eqref{hkl6} and \eqref{hkl7} are reminiscent of the conditions in Erd\H{o}s' paper \cite{Erdos}, except that no numerators were present in that paper, and it had $K=1$ and $a_n=a_{1,n}$. 
Nonetheless, the role played by these conditions is similar and will be used to show that the theorem follows from a result about the partial sums $\sum_{n=N}^{\infty} b_{k,n}/a_{k,n}$ when $N$ is large.
Assumptions \eqref{hkl5} and \eqref{hkl6} arise from considering several numbers and from having a non-fixed denominators, respectively, and essentially allow us to replace the sequences $\{a_{k,n}\}_{n=1}^\infty$ and $\{b_{k,n}\}_{n=1}^\infty$ with the single sequence $\{a_{n}\}_{n=1}^\infty$ for a central part of the proof.

The novelty of our result lies in the origins of assumptions \eqref{hkl1} and \eqref{hkl2}. As we shall see, their role is to ensure that no integer polynomial in $K$ variables will vanish at infinitely many of the partial sums of the numbers $\alpha_k$.

From the conditions in Theorem \ref{thm:general}, we could easily formulate conditions under which \eqref{hkl2} and \eqref{hkl7} hold for every value of $d \in \mathbb{N}$. This immediately leads us to a criterion for algebraic independence of the $\alpha_k$ in the following way.

\begin{Theorem}
	\label{thm:alg.indep.}
	Let $K$ be a positive integer, let $p$ be a prime number, and let $\varepsilon$ and $\kappa$ be positive real numbers with $\kappa <1$.
	For $k=1,\dots ,K$, let $\{a_{k,n}\}_{n=1}^\infty$ and $\{b_{k,n}\}_{n=1}^\infty$ be sequences of non-zero integers with $p\nmid \gcd(a_{k,n},b_{k,n})$ so that assumption \eqref{hkl1} is satisfied for exactly one $n\le N$ when $N$ is sufficiently large, while
	\begin{equation}\label{eq:np_akn_ration to infty}
		\lim_{N\to\infty} \frac {\max_{n=1,\dots ,N}  \nu_p(a_{k,n})} {1+ \max_{n=1,\dots ,N} \nu_p(a_{k-1,n})}=\infty,
	\end{equation} 
	where $a_{0,n} =1$ for each $n\in\bN$.
	Suppose there is a non-decreasing sequence $\{a_n\}_{n=1}^\infty$ of positive integers that, for each $k=1,\dots ,K$ and all sufficiently large $n,A\in\bN$, satisfies assumption \eqref{hkl4}, \eqref{hkl6},
	\begin{equation}\label{hkl5'}
		a_{n} 2^{-(\log_2 a_{n})^\kappa} \le |a_{k,n}| \leq \max \big\{a_{n} 2^{(\log_2 a_{n})^\kappa}, 2^{A^n}\big\},
	\end{equation}
	and
	\begin{equation}\label{eq:limsup a_n A-n}
		\limsup_{n\rightarrow\infty} a_{n}^{\frac{1}{A^n}} =
		\infty.
	\end{equation}
	For $k=1,\dots ,K$, set $\alpha_k= \sum_{n=1}^\infty \frac {b_{k,n}}{a_{k,n}}$. 
	Then $\alpha_1,\dots ,\alpha_K$ are algebraically independent over $\mathbb Q$. 
\end{Theorem}

As stated in the results, when $K=2$, most of the conditions needed for our two previous results become simpler at the cost of removing potential degenerate cases of algebraic dependence. Indeed, it is straightforward to check in this case that the following holds.

\begin{Theorem}
	\label{thm:alg.indep. K=2}
	Let $\{a_{k,n}\}_{n=1}^\infty$, $\{b_{k,n}\}_{n=1}^\infty$, $\{a_{n}\}_{n=1}^\infty$, and $\alpha_k$ be given as in Theorem \ref{thm:general} with $K=2$, $d=3$,
	and the additional assumption that for every positive integer $A$,
	\begin{equation*}
		\limsup_{n\rightarrow\infty} a_{n}^{\frac{1}{A^n}} =
		\infty.
	\end{equation*}
	Then $\alpha_1$ and $\alpha_2$ are non-degenerately algebraically independent over $\mathbb Q$. 
\end{Theorem}

If we focus on $K=1$, we also get new irrationality and transcendence criteria, which we will phrase in terms of the following definitions of $p$-irrationality and $p$-transcendence.
\begin{Definition}
	Let $p$ be a prime number. 
	We say that a sequence of non-zero numbers $\{a_n\}_{n=1}^\infty$ is $p$-irrational or $p$-transcendental, if for every sequence $\{c_n\}_{n=1}^\infty$ of integers $c_n$ with $p\nmid c_n$, the infinite series
	\begin{align*}
		\sum_{n=1}^{\infty}\frac{1}{a_n c_n}
	\end{align*}
	converges to a number that is irrational or transcendental, respectively.
\end{Definition}

This generalises the notions of irrational and transcendental sequences introduced by Erd\H{o}s \cite{Erdos} and Han\v{c}l \cite{hanTransc}, respectively, where it is assumed that $c_n>0$ rather than $p \nmid c_n$.
From Theorem \ref{thm:general}, we obtain the below result.
\begin{Theorem}\label{thm:p-irrational}
	Let $p$ be a prime number, let $C$ be a positive integer, and let $\varepsilon$ and $\kappa$ be positive real numbers with $\kappa<1$.
	Let $\{a_n\}_{n=1}^\infty$ and $\{b_n\}_{n=1}^\infty$ be sequences of non-zero integers with $p\nmid \gcd(a_{n},b_{n})$ such that $\limsup_{n\to\infty} \nu_p(a_n)=\infty$ and either $\nu_p(a_m)\ne\nu_p(a_n)$ or $\nu_p(a_m)=\nu_p(a_n)\le C$ when $m\ne n$.
	Suppose that $a_{n+1}\ge a_n\ge n^{1+\varepsilon}$ and $|b_n|\le 2^{(\log_2 |a_n|)^\kappa}$ when $n$ is sufficiently large.
	Then the sequence $\{a_n/b_n\}$ is $p$-irrational if
	\begin{equation}\label{eq:limsup:irr}
		\limsup_{n \rightarrow \infty} |a_n|^{\frac{1}{2^n}} = \infty,
	\end{equation}
	and it is $p$-transcendental if equation \eqref{eq:limsup a_n A-n} is satisfied for every positive integer $A$.
\end{Theorem}
\begin{proof}
	Let $c_n\in \bZ$ with $p\nmid c_n$, and let $\sigma:\bN\to\bN$ be the bijection so that $|A_n|$ is non-decreasing where $A_n=a_{\sigma(n)} c_{\sigma(n)}$, and put $B_n = b_{\sigma(n)}$.
	If equation \eqref{eq:limsup a_n A-n} is satisfied, let $d\in\bN$ be arbitrary.
	If not, assume equation \eqref{eq:limsup:irr}, and put $d=1$.
	Since certainly $
	\sum_{n=1}^{\infty}\frac{b_n}{a_n c_n} = \sum_{n=1}^{\infty}\frac{B_n}{A_n},$
	the theorem follows if we can show that $\{A_n\}_{n=1}$ and $\{B_n\}_{n=1}^\infty$ satisfy the assumptions of Theorem \ref{thm:general}.
	We immediately notice that inequalities \eqref{hkl2}, \eqref{hkl5}, \eqref{hkl6}, and \eqref{hkl7} are satisfied.
	Since $A_n$ is non-decreasing and $|a_n c_n|\ge a_n \ge n^{1+\varepsilon}$, inequality \eqref{hkl4} must also be satisfied.
	As $p\nmid c_n, \gcd(a_n,b_n)$, it follows that $p\nmid \gcd(A_n, B_n)$, while the facts that $p\nmid c_n$, $\limsup_{n \rightarrow \infty} \nu_p(n)=\infty$, and $\nu_p(n)\ne\nu_p(m)$ when $\nu_p(n)$ and $m\ne n$ are large ensure that equation \eqref{hkl1} can be satisfied by at most one $n\le N$ when $N$ is large.
	Thus, all assumptions of Theorem \ref{thm:general} are satisfied.
\end{proof}
\begin{Remark}\label{rem:p-irrational}
	When $\{c_n\}_{n=1}^\infty$ is a sequence of non-zero integers, as is the case when all $c_n$ are non-divisible by any given prime, then $\alpha = \sum_{n=1}^{\infty} \frac{b_n}{a_n c_n}$ is absolutely convergent, and so the resulting number is the same if we reorder the terms.
	Hence, if we have $\nu_p(a_m)\ne \nu_p(a_n)$ for all $m\ne n$, we can freely reorder the terms to ensure that $\nu_p(a_n)$ is strictly increasing.
	Then $a_n\ge p^{\nu_p(a_n)}\ge p^{n-1}$, and so we automatically get the assumptions $\limsup\nu_p(a_n)=\infty$ and $a_n\ge n^{1+\varepsilon}$ when $n$ is sufficiently large, thus reducing the number of assumptions one needs to check when applying Theorem \ref{thm:p-irrational} to such a sequence.
\end{Remark}

In Theorem \ref{thm:general}, we might also consider what happens if we keep track of multiple prime numbers $p$ rather than just one.
This leads to the below theorem, the proof of which is almost identical to that of Theorem \ref{thm:general}.
As such, these two theorems will be proven in unison.

\begin{Theorem}
	\label{thm:general infinite}
	Let $K$ and $d$ be positive integers, let $\varepsilon$ and $\kappa$ be positive real numbers with $\kappa <1$ and let
	$\mathcal{P}$ be a set of infinitely many prime numbers.
	For $k=1,\ldots, K$, let $\{a_{k,n}\}_{n=1}^\infty$ and $\{b_{k,n}\}_{n=1}^\infty$ be sequences of non-zero integers with $p\nmid \gcd(a_{k,n},b_{k,n})$ for each fixed $p\in\cP$.
	Suppose that for each sufficiently large $N\in\bN$, exactly one $n\le N$ satisfies equation \eqref{hkl1} while
	\begin{equation}\label{hkl3}
		\max_{n=1,\dots ,N}\nu_p(a_{k,n}) > d \max_{n=1,\dots ,N}\nu_p(a_{k-1,n}),
	\end{equation}
	where $a_{0,n}=1$.
	Suppose that there is a sequence $\{a_n\}_{n=1}^\infty$ that satisfies assumptions \eqref{hkl4}, \eqref{hkl5}, \eqref{hkl6}, and \eqref{hkl7} for every $k=1,\dots ,K$ and $n\in\bN$.
	For $k=1,\ldots, K$, set
	$\alpha_k= \sum_{n=1}^\infty \frac{b_{k,n}}{a_{k,n}}$. 
	Then $(\alpha_1,\dots ,\alpha_K)$ is not the root of any non-zero polynomial of $K$ variables with integer coefficients and degree less than or equal to $d$.

	If $K=2$, this is also true when $d$ is replaced by 3 in equation \eqref{hkl3} and the conclusion is replaced by non-degenerate independence of order $d$.
\end{Theorem}

As with Theorem \ref{thm:general}, one could replace $3$ in the last part by a number $d' > 3$ to cut down on the number of potentially `bad' polynomials

This theorem has corollaries in complete parallel to Theorems \ref{thm:alg.indep.} and \ref{thm:alg.indep. K=2}.
\begin{Theorem}\label{thm:alg.indep. infinite}
	Let $K$ be a positive integer, let $\cP$ be an infinite set of a prime numbers, and let $\varepsilon$ and $\kappa$ be positive real numbers with $\kappa <1$.
	For $k=1,\dots ,K$, let $\{a_{k,n}\}_{n=1}^\infty$ and $\{b_{k,n}\}_{n=1}^\infty$ be sequences of non-zero integers such that, for all $p\in\cP$, $p\nmid \gcd(a_{k,n},b_{k,n})$, assumption \eqref{hkl1} is satisfied for exactly one $n\le N$ when $N$ is sufficiently large, and if $k>1$, then equation \eqref{eq:np_akn_ration to infty} is satisfied.
	Suppose there is a non-decreasing sequence $\{a_n\}_{n=1}^\infty$ of positive integers that, for each $k=1,\dots ,K$ and all sufficiently large $A,n\in\bN$, satisfies assumptions \eqref{hkl4}, \eqref{hkl6}, \eqref{hkl5'}, and \eqref{eq:limsup a_n A-n}.
	For $k=1,\dots ,K$, set $\alpha_k= \sum_{n=1}^\infty \frac {b_{k,n}}{a_{k,n}}$. 
	Then $\alpha_1,\dots ,\alpha_K$ are algebraically independent over $\mathbb Q$. 
\end{Theorem}
\begin{Theorem}
	\label{thm:alg.indep. K=2 infinite}
	Let $\{a_{k,n}\}_{n=1}^\infty$, $\{b_{k,n}\}_{n=1}^\infty$, $\{a_{n}\}_{n=1}^\infty$, and $\alpha_k$ be given as in Theorem \ref{thm:general infinite} with $K=2$, $d=3$,
	and the additional assumption that for every positive integer $A$,
	\begin{equation*}
		\limsup_{n\rightarrow\infty} a_{n}^{\frac{1}{A^n}} =
		\infty.
	\end{equation*}
	Then $\alpha_1$ and $\alpha_2$ are non-degenerately algebraically independent over $\mathbb Q$. 
\end{Theorem}
\begin{Remark}\label{rem:thms alg.indep.}
	Theorems \ref{thm:alg.indep.} and \ref{thm:alg.indep. infinite} are very similar in their assumptions.
	In fact, the only differences are that Theorem \ref{thm:alg.indep. infinite} does not assume equation \eqref{eq:np_akn_ration to infty} for $K=1$ and that it makes its assumptions for infinitely many primes $p$.
	This is a consequence of how an unbounded $d$ makes $\nu_p(a_{k,n})$ unbounded when $k>1$, which in turn makes the differences between inequality \eqref{hkl3} and equation \eqref{hkl2} small when $d$ grows large.
	As such, Theorem \ref{thm:alg.indep. infinite} should only be used when $\nu_p(a_{n,1})$ is bounded for all fixed $p\in\cP$ (or when it is unknown if this is the case); otherwise, Theorem \ref{thm:alg.indep.} would be sufficient by just picking any one $p$ from $\cP$ with unbounded $\nu_p(a_{1,n})$.
\end{Remark}
Aiming for a result corresponding to Theorem \ref{thm:p-irrational}, we introduce notions of $\cP$-irrationality and $\cP$-transcendence when $\cP$ is a set of prime numbers. 

\begin{Definition}
	Let $\cP$ be a set of prime numbers.
	We say that a sequence $\{a_n\}_{n=1}^\infty$ is $\cP$-irrational or $\cP$-transcendental if the series $\sum_{n=1}^{\infty} \frac{1}{a_n c_n}$ converges to an irrational or transcendental number, respectively, for all sequences of integers $\{c_n\}_{n=1}^\infty$ with $p\nmid c_n$ for all $p\in\cP$.
\end{Definition}

By considerations similar to those in the last part of Remark \ref{rem:thms alg.indep.}, it is clear that when $K=1$, Theorem \ref{thm:general infinite} is only interesting when $\nu_p(a_{1,n})$ is bounded for all fixed $p$, since one could otherwise apply Theorem \ref{thm:general} with much less effort instead.
For that reason, we formulate the result corresponding to Theorem \ref{thm:p-irrational} as follows.


\begin{Theorem}\label{thm:P-irrational}
	Let $\cP$ be a set of infinitely many prime numbers, and let $\varepsilon$ and $\kappa$ be positive real numbers with $\kappa<1$.
	Let $\{a_n\}_{n=1}^\infty$ and $\{b_n\}_{n=1}^\infty$ be sequences of non-zero integers such that, for each fixed $p\in\cP$, $p\nmid \gcd(a_{n},b_{n})$ for all $n$, while $\nu_p(a_n)$ is bounded and attains its maximum value exactly once.
	Suppose that $a_{n+1}\ge a_n\ge n^{1+\varepsilon}$ and $|b_n|\le 2^{(\log_2 |a_n|)^\kappa}$.
	Then the sequence $\{a_n/b_n\}$ is $\cP$-irrational if equation \eqref{eq:limsup:irr} is satisfied, and it is $\cP$-transcendental if equation \eqref{eq:limsup a_n A-n} is satisfied for every positive integer $A$.
\end{Theorem}
\begin{proof}
	The proof is essentially the same as that of Theorem \ref{thm:p-irrational}, except that we apply Theorem \ref{thm:general infinite} rather than Theorem \ref{thm:general}.
	As such, our interest in equation \eqref{hkl2} is replaced with inequality \eqref{hkl3}.
\end{proof}

\section{Examples and applications}

In this section, we give a number of examples of applications of our results. 
It is of interest that the terms of our series have varying signs since this was not allowed in previous papers, such as \cite{ak,Erdos,hanTransc,hanLinIndep,l}.
Let $f_1,\ldots, f_K$ denote arbitrary functions from $\bN$ into $\bZ$.
This way, we can express the varying signs as $(-1)^{f_k(n)}$.
Possible choices for $f_k(n)$ could, for instance, be the product $kn$, the number of divisors of $n$, or the Euler totient function of $n$, i.e., the number of $m=1,\ldots,n$ such that $\gcd(m,n)=1$.

We start by giving some simple examples to highlight the immediate applications of our results in terms of algebraic independence.
For these examples, use Theorem \ref{thm:alg.indep.}.

\begin{Example}
	\label{ex:thm2 n^k}
	Let $\{a_n\}_{n=1}^\infty$ be a non-decreasing sequence of positive odd integers such that for every positive integer $A$,
	\begin{equation*}
		\limsup_{n\rightarrow\infty} a_n^{\frac{1}{A^n}} =
		\infty .
	\end{equation*}
	Set
	$\alpha_k= \sum_{n=1}^\infty \frac {(-1)^{f_k(n)}}{a_n2^{n^{k}}}$, $k=1,\dots ,K$. Then the numbers $\alpha_1,\dots ,\alpha_K$ are algebraically independent over $\mathbb Q$.
\end{Example}


\begin{Example}\label{ex:thm2 nu2n}
	Let $K$ and $z$ be positive integers with $z\ge 2$, and let $\{a_n\}_{n=1}^\infty$ be a non-decreasing sequence of integers with $\gcd(a_n,z)=1$, $a_n\ge n^{1+\varepsilon}$, and
	\begin{equation*}
		\limsup_{n\rightarrow\infty} a_n^{\frac{1}{A^n}} =
		\infty,
	\end{equation*}
	for every positive integer $A$.
	Set
	$\alpha_k= \sum_{n=1}^\infty \frac {(-1)^{f_k(n)}}{a_n z^{\nu_2(n)^{k}}}$, $k=1,\dots ,K$. Then the numbers $\alpha_1,\dots ,\alpha_K$ are algebraically independent over $\mathbb Q$. 
\end{Example}
As is seen from the theorems, we have non-degenerate algebraic independence with weaker restrictions when we know that $K=2$, which gives us the below example by applying Theorem \ref{thm:alg.indep. K=2}.
\begin{Example}\label{ex:thm3}
	Let $z\ge 2$ be a positive integer, and let $\{a_n\}_{n=1}^\infty$ be a non-decreasing sequence of positive integers with $\gcd(z,a_n)=1$ and, for all positive integers $A$, $\limsup a_n^{1/A^n} = \infty$.
	For $k=1,2,\ldots$, set $\alpha_k = \sum_{n=1}^{\infty}\frac{(-1)^{f_k(n)}}{a_n z^{(3^k -1)n}}$.
	If $k\ne l$, then $\alpha_k$ and $\alpha_l$ are non-degenerately algebraically independent.
\end{Example}

%
%
Meanwhile, by using the more general Theorem \ref{thm:general}, we are also able to prove irrationality and linear independence of series where we are not currently able to determine algebraic independence, as seen in the next example.
\begin{Example}\label{ex:thm1}
	Let $K$ and $z$ be positive integers with $z\ge 2$, and let $\{a_n\}_{n=1}^\infty$ be a non-decreasing sequence of positive integers such that $\gcd(a_n,z)=1$ and $\limsup a_n^{1/(K+1)^n} = \infty$.
	For $k=1,\ldots, K$, set $\alpha_k =\sum_{n=1}^{\infty} \frac{(-1)^{f_k(n)}}{a_n z^{k\, \nu_p(n)}}$.
	Then $1,\alpha_1,\ldots,\alpha_K$ are linearly independent over $\bQ$.
\end{Example}	
\begin{Remark}\label{rem:zeta5}
	We are not able to prove that the number $\zeta(5)=\sum_{n=1}^\infty \frac 1{n^5}$ is irrational. Erd\H{o}s \cite{Erdos} proved that if $ \{a_n\}_{n=1}^\infty$ is a non-decreasing sequence of positive integers such that $
	\limsup_{n\rightarrow\infty} a_n^{\frac{1}{2^n}} =
	\infty ,$
	then
	$\alpha= \sum_{n=1}^\infty \frac 1{a_nn^5}$ is an  irrational number. 
	From Theorem \ref{thm:general} with $p=2$, we now obtain that if all $a_n$ are also odd, then the number $\sum_{n=1}^{\infty} \frac{(-1)^{f(n)}}{a_n n^5}$
	is irrational for all functions $f:\bN\to\bZ$.
\end{Remark}
\begin{Remark}\label{rem:zetaf,g}
	For integer functions $f:\bN\to\bZ$ and $g(n):\bN\to\bN$, write $\zeta_{f,g} (k) = \sum_{n=1}^{\infty} \frac{(-1)^{f(n)}}{g(n) n^k}$.
	Then $\zeta_{0,1}$ is Riemann's zeta function, where $0$ and $1$ denote the functions that are constantly $0$ and $1$, respectively.
	We get from Remark \ref{rem:zeta5} that $\zeta_{f,g}(5)$ is irrational when $\{g(n)\}_{n=1}^\infty$ defines a non-decreasing sequence of positive odd integers with $\limsup_{n\to\infty} g(n)^{1/2^n}=\infty$, regardless of $f$.
	In fact, $\zeta_{f,g}(k)$ is irrational for all such $g$ and all integers $k\ge 2$.
	However, if we remove the conditions that $g(n)$ is non-decreasing and that $\limsup_{n\to\infty} g(n) = \infty$ and fix $f(n)$ and $k\ge 2$, then we can pick $g$ so that $\zeta_{f,g}(k)$ is rational.
	If $f(n)$ is both infinitely often even and infinitely often odd, we can take a permutation $\sigma:\bN\to\bN$ such that
	\begin{align*}
		\zeta_{f,g}(k) = \sum_{n=1}^{\infty}\frac{(-1)^{f(n)}}{n^kg(n)}
		= \sum_{n=1}^{\infty}\frac{(-1)^{n}}{\sigma(n)^k g(\sigma(n))}.
	\end{align*}
	Then picking $g$ so that $g(\sigma(2m))=\sigma(2m-1)^k$ and $g(\sigma(2m-1))=\sigma(2m)^k$, we get $\zeta_{f,g}(k)=0$, which is rational.
	If $f(n)$ is even only finitely often, then $\zeta_{f,g}(k)$ is a rational number minus the number $\sum_{n=N}^{\infty}\frac{1}{n^k g(n)}$.
	By a result due to Han\v{c}l \cite{hanExprReal}, we can now choose $g$ so that this number is rational.
	If $f(n)$ is odd only finitely often, we just note that $\zeta_{f,g}(k)=-\zeta_{f+1,g}(k)$. 
	Then $1+f(n)$ is even only finitely often, and we can pick $g$ so that $\zeta_{f,g}(k)$ is rational by the previous consideration.
	
\end{Remark}
\begin{Open Problem}
	While it is well-known that $\zeta_{0,1}(2)=\zeta(2)=\pi^2/6$ is transcendental, and Apéry \cite{apery} showed that $\zeta_{0,1}(3)=\zeta(3)$ is irrational, the authors do not know if $\zeta_{f,1}(k) = \sum_{n=1}^{\infty}\frac{(-1)^{f(n)}}{n^k}$ is irrational in general for $k\ge 2$.
\end{Open Problem}

We now give examples of $p$-irrationality and $p$-transcendence of a sequence, using Theorem \ref{thm:p-irrational} and Remark \ref{rem:p-irrational}.
\begin{Example}\label{ex:p-irr}
	\label{hankrislauEx11}
	Let $ \{a_n\}_{n=1}^\infty$ be a non-decreasing sequence of positive odd integers with $\limsup_{n\rightarrow\infty} a_n^{1/2^n} =
	\infty .$
	Let $ \{r_n\}_{n=1}^\infty$ be a sequence of pairwise different non-negative integers.
	Then the sequence $ \{2^{r_n}a_n\}_{n=1}^\infty$ is 2-irrational. Likewise, for any integer $z \ge 2$, if we assume $\gcd(a_n,z)=1$ in place of $a_n$ being odd, then the sequence $ \{z^{r_n} a_n \}_{n=1}^\infty$ is $p$-irrational for all primes $p$ dividing $z$.
\end{Example}

\begin{Example}\label{ex:p-transc}
	Let $z\ge 2$ be a positive integer, and let $ \{a_n\}_{n=1}^\infty$ be a non-decreasing sequence of positive integers coprime with $z$ such that, for all $A\in\bN$,
	$\limsup_{n\rightarrow\infty} a_n^{1/A^n} =
	\infty .$
	Let $ \{r_n\}_{n=1}^\infty$ be a sequence of pairwise different non-negative integers.
	Then the sequence $ \{z^{r_n}a_n\}_{n=1}^\infty$ is $p$-transcendental for all primes $p$ dividing $z$.
\end{Example}

Finally, we present the below four examples of Theorems \ref{thm:general infinite}, \ref{thm:alg.indep. infinite}, \ref{thm:alg.indep. K=2 infinite}, and \ref{thm:P-irrational}, respectively.
Comparing these with Examples \ref{ex:thm1}, \ref{ex:thm2 n^k}, \ref{ex:thm3}, and \ref{ex:p-irr}, respectively, we see some of the differences in applicability between considering a single prime or infinitely many.


\begin{Example}
	Let $K$ be a positive integer, and let $ \{r_{n}\}_{n=1}^\infty$ be a strictly increasing sequence of positive integers such that 
	\begin{equation*}
		\limsup_{n\to\infty} \frac{r_n}{(K+1)^n} >0.
	\end{equation*}
	Let $\{p_n\}_{n=1}^\infty$ be an unbounded and non-decreasing sequence of prime numbers.
	For $k=1,\ldots, K$, set $\alpha_k = \sum_{n=1}^{\infty}\frac{(-1)^{f_k(n)}}{p_n^{k+r_n}}$.
	Then the numbers $1,\alpha_1,\ldots,\alpha_K$ are linearly independent over $\bQ$.
\end{Example}

\begin{Example}
	Let $\{p_n\}_{n=1}^\infty$ be a strictly increasing sequence of odd prime numbers, and let $\{r_n\}_{n=1}^\infty$ be a non-decreasing sequence of integers such that $r_n\ge p_n$ and $\limsup_{n\to\infty} (r_n / A^n) = \infty$ for all $A\in\bN$.
	Set 
	$\alpha_1 = \sum_{n=1}^{\infty}\frac{(-1)^{f_1(n)}}{ 2^{r_n} p_n }$
	and 
	$\alpha_{k+1} = \sum_{n=1}^{\infty}\frac{(-1)^{f_k(n)}}{ 2^{r_n} (p_1\cdots p_n)^{n^{k-1}} }$
	for each $k\in\bN$ with $k>1$. Then $\alpha_1,\ldots,\alpha_K$ are algebraically independent for all $K\in\bN$.
	
	For this example, it may not be clear that we can actually pick a sequence $\{a_n\}_{n=1}^\infty$ of integers that satisfies \eqref{hkl5'} for each $k$ and all large values of $A$ and $n$.
	Write $a_{1,n}=2^{r_n}p_n$ and $a_{k,n} = 2^{r_n}(p_1\cdots p_n)^{n^{k-1}}$ for $k>1$, fix $K$, and pick $a_n = 2^{r_n}$ and $\kappa=1/2$.
	Then the only assumption that is not trivially satisfied is the upper bound of \eqref{hkl5'} when $k>1$.
	Since $a_{1,n}<\cdots<a_{K,n}$, it suffices to prove this when $A=2$, $k=K>1$ and $a_{K,n}>2^{2^n}$.
	By using $\log_2$, the upper bound of \eqref{hkl5'} is equivalent to
	\begin{align*}
		r_n + \log_2\big((p_1\cdots p_n)^{n^k}\big)
		=\log_2 a_{1,n}
		\le \log_2 \big(a_n 2^{(\log_2 a_n)^\kappa}\big) 
		= r_{n} + r_n^{1/2},
	\end{align*}
	i.e.,
	\begin{equation*}
		r_n^{1/2} \ge \log_2\big((p_1\cdots p_n)^{n^k}\big) = n^k \sum_{m=1}^{n} \log_2 p_m.
	\end{equation*}
	If $r_n\ge n^{4K}$, then this is indeed satisfied, recalling that $r_n\ge p_n$ and calculating
	\begin{align*}
		r_n^{1/2} = r_n^{1/4} r_n^{1/4} 
		\ge (n^{4K})^{1/4} p_n^{1/4} 
		> n^{K-1} n \log_2 p_n
		\ge n^{K-1} \sum_{m=1}^{n}\log_2 p_m.
	\end{align*}
	If $p_n\ge n^{4K}$, we are therefore done, so consider the case $p_n<n^{4K}$.
	Since $a_{K,n}>2^{2^n}$, we then have
	\begin{align*}
		2^{ 2^n} < a_{K,n} 
		= 2^{r_n}(p_1\cdots p_n)^{n^{K-1}}
		< 2^{r_n} p_n^{n^K}
		< 2^{r_n} n^{4K n^K}. 
	\end{align*}
	As this clearly ensures that $r_n\ge n^{4K}$ for all large values of $n$, we are done.
\end{Example}

\begin{Example}
	Let $\{p_n\}_{n=1}^\infty$ be a strictly increasing sequence of odd prime numbers, and let $ \{r_{n}\}_{n=1}^\infty$ be a non-decreasing sequence of integers such that $r_n\ge p_n$ and $\limsup_{n\to\infty} (r_n/A^n) = \infty$ for all positive integers $A$.
	For $k=1,2,\ldots$, set $\alpha_k = \sum_{n=1}^{\infty}\frac{(-1)^{f_k(n)}}{2^{r_n} p_n^{3^k-1}}$.
	If $k\ne l$, then $\alpha_k$ and $\alpha_l$ are non-degenerately algebraically independent.
\end{Example}
\begin{Example}
	\label{hankrislauEx12}
	Let $\{p_n\}_{n=1}^\infty$ be an increasing sequence of prime numbers, let $\cP$ be an infinite subset of $\{p_n: n\in\bN\}$, and let $\{a_n\}_{n=1}^\infty$ be an increasing sequence of positive integers with $\limsup_{n\rightarrow\infty}(a_n p_n)^{1/2^n} =\infty$ and $p\nmid a_n$ for all $p\in\cP$.
	Then the sequences $\{a_n p_n\}_{n=1}^\infty$ and $\{a_n p_n^2 p_{n-1}\cdots p_1\}_{n=1}^\infty$ are each $\cP$-irrational.
	If furthermore, $\limsup_{n\rightarrow\infty}(a_n p_n)^{1/A^n} =\infty$
	for all positive integers $A$, then the sequences $\{a_n p_n\}_{n=1}^\infty$ and $\{a_n p_n^2 p_{n-1}\cdots p_1\}_{n=1}^\infty$ are each $\cP$-transcendental.
\end{Example}

\section{Preliminaries}\label{sec:prelim}

In this section, we give some auxiliary  results needed for the proofs of the main theorems. 
The first one is a slight strengthening of Lemma 5 of \cite{hns} when specialized to a single sequence $\{a_n/b_n\}$ with $b_n\le 2^{(\log_2 a_n)^\kappa}$.
\begin{Lemma}\label{lemma:hns}
	Let $\varepsilon>0$, $0<\kappa<1$, and $M\ge 0$.
	Let $\{a_n\}_{n=1}^\infty$ and $\{b_n\}_{n=1}^\infty$ be sequences of positive integers such that $a_n$ is non-decreasing and
	\begin{equation}\label{eq:limsup}
		\limsup_{n \rightarrow \infty} a_n^{ (M+ 2)^{-n} } = \infty.
	\end{equation}
	Suppose that for all sufficiently large $n$ that
	\begin{align*}
		a_n\ge n^{1+\varepsilon}
	\end{align*}
	and
	\begin{align*}
		b_n\le 2^{(\log a_n)^\kappa}.
	\end{align*}
	Then for any fixed $0<c<1$,
	\begin{equation*}
		\liminf_{N \rightarrow \infty}  2^{N^2 (\log_2 a_{N-1})^c} \left( \prod_{n=1}^{N-1} a_n\right)^{M+1} \sum_{n=N}^\infty \left\vert \frac{b_n}{a_n}\right\vert = 0.
	\end{equation*}
\end{Lemma}
Not surprisingly, the proof will closely follow that of Lemma 5 of \cite{hns} and use the same preliminary lemmas, below, which are proven in \cite{hns}.
In their original form, the lemmas had additional assumptions, but these were never used in their proofs and are thus omitted here.
They may also be extracted from the proof in \cite{Erdos}.
\begin{Lemma}\label{lem:sumBound general}
	Let $\{a_n\}_{n=1}^\infty$ and $\{b_n\}_{n=1}^\infty$ satisfy the assumptions of Lemma \ref{lemma:hns}.
	Then there exists a number $\gamma>0$ that does not depend on $N$ and such that for all sufficiently large $N$,
	\begin{equation*}
		\sum_{n=N}^{\infty} \frac{b_n}{a_n} \le \frac{1}{a_N^\gamma}.
	\end{equation*}
\end{Lemma}
\begin{Lemma}\label{lem:sumBound an>2n}
	Let $\{a_n\}_{n=1}^\infty$ and $\{b_n\}_{n=1}^\infty$ satisfy the assumptions of Lemma \ref{lemma:hns}.
	Suppose that $a_n\ge 2^n$ for all sufficiently large $n$.
	Then there is a fixed $0<\Gamma<1$ such that for all sufficiently large $N$,
	\begin{equation*}
		\sum_{n=N}^{\infty} \frac{b_n}{a_n} \le \frac{2^{(\log_2 a_N)^\Gamma}}{a_N}.
	\end{equation*}
\end{Lemma}
\begin{Lemma}\label{lem:sumBound an>2n sometimes}
	Let $\{a_n\}_{n=1}^\infty$ and $\{b_n\}_{n=1}^\infty$ satisfy the assumptions of Lemma \ref{lemma:hns}.
	Then there is a fixed number $0<\Gamma<1$ so that if $N$ and $Q$ are sufficiently large and $a_n\ge 2^n$ for $n=N,\ldots, Q$, then
	\begin{equation*}
		\sum_{n=N}^{Q} \frac{b_n}{a_n}\le \frac{2^{(\log_2 a_N)^\Gamma}}{a_N}.
	\end{equation*}
\end{Lemma}
\begin{Lemma}\label{lem:>1+1/N2}
	Let $\{y_n\}_{n=1}^\infty$ be an unbounded sequence of positive real numbers.
	Then there are infinitely many $N$ such that
	\begin{equation*}
		y_N > \bigg(1+\frac{1}{N^2}\bigg)\max_{1\le n< N} y_n.
	\end{equation*}
\end{Lemma}
By a simple induction argument, we notice that for $k<N$ and $\delta\ge 0$,
\begin{align}\nonumber
	(M+2+\delta)^{N} &= (M+2+\delta)^{N-1} + (M+1+\delta)(M+2+\delta)^{N-1}=\cdots 
	\\\label{eq:boundM+1+delta}
	&= (M+2+\delta)^{k} + (M+1+\delta)\sum_{n=k}^{N-1}(M+2+\delta)^{n},
\end{align}
which will be used for the proof of Lemma \ref{lemma:hns}.
Equation \eqref{eq:boundM+1+delta} also helps prove the below corollary to Lemma \ref{lem:>1+1/N2}.
\begin{Corollary}\label{cor:>1+1/N2}
	Let $\{a_n\}_{n=1}^\infty$ and $\{b_n\}_{n=1}^\infty$ satisfy the assumptions of Lemma \ref{lemma:hns}.
	Let $k$ be a positive integer. 
	Then for infinitely many $N>k$,
	\begin{equation*}
		a_N > \bigg(1 + \frac{1}{N^2}\bigg)^{(M+2)^N}\Big(\max_{k\le n< N} a_n^{(M+2)^{-n}}\Big)^{(M+2)^N}.
	\end{equation*}
	For such $N$, we further have
	\begin{equation*}
		a_N > \bigg(1 + \frac{1}{N^2}\bigg)^{(M+2)^N} \prod_{n=k}^{N-1}a_n^{M+1}.
	\end{equation*}
\end{Corollary}
\begin{proof}
	The first inequality follows immediately from Lemma \ref{lem:>1+1/N2} by taking $y_n = a_n^{(M+2)^{-n}}$ for $n\ge k$ and $y_n=y_k$ for $n<k$.
	We then use \eqref{eq:boundM+1+delta} to conclude
	\begin{align*}
		\Big(\max_{k\le n< N} a_n^{(M+2)^{-n}}\Big)^{(M+2)^N} 
		&\ge \Big(\max_{k\le n< N} a_n^{(M+2)^{-n}}\Big)^{(M+1)\sum_{n=k}^{N-1} (M+2)^n}
		\\&
		\ge \prod_{n=k}^{N-1}a_n^{M+1}.
		\qedhere
	\end{align*}
\end{proof}

\begin{proof}[Proof of Lemma \ref{lemma:hns}]
	To shorten notation, write
	\begin{align*}
		Z_N = 2^{N^2 (\log_2 a_{N-1})^c}\Bigg(\prod_{n=1}^{N-1} a_n\Bigg)^{M+1} \sum_{n=N}^{k_2}\frac{b_n}{a_n}.
	\end{align*}
	We here split the proof into two cases depending on whether
	\begin{equation}\label{eq:diverges fast}
		\limsup_{n\to\infty} a_n^{\left(M+2+\delta\right)^{-n}} = \infty
	\end{equation}
	is true some fixed $\delta>0$.
	
	
	\textbf{Case 1 (equation \eqref{eq:diverges fast} holds for some $\delta>0$).}
	Pick $0<\gamma<1$ as in Lemma \ref{lem:sumBound general}, and let $z>2$ be some sufficiently large number.
	Pick $k_1, k_2, N\in\bN$ as follows.
	Let $k_2$ be the smallest integer such that
	\begin{equation}\label{eq:case1k2}
		a_{k_2}^{\left(M+2+\delta\right)^{-k_2}} > z^{1/\gamma},
	\end{equation}
	let $k_1$ be the largest integer such that $k_1 < k_2$ and
	\begin{equation}\label{eq:case1k1}
		a_{k_1} \le z^{k_1},
	\end{equation}
	and let $N$ be the smallest number such that $N>k_1$ and
	\begin{equation}\label{eq:case1N}
		a_N^{\left(M+2+\delta\right)^{-N}} \ge z.
	\end{equation}
	Note that $N\le k_2$ and $N\to\infty$ as $z\to\infty$.
	By inequalities \eqref{eq:case1k1} and \eqref{eq:case1N}, $a_{n} < z^{(M+2+\delta)^{n}}$ when $k_1\le n<N$.
	Hence,
	\begin{equation}\label{eq:case1-2logBound}
		2^{N^2 (\log_2 a_{N-1})^c} \le 2^{N \left(M+2+\delta\right)^{(N-1)c} (\log_2 z)^c}
		< z^{N^2 \left(M+2+\delta\right)^{Nc}},
	\end{equation}
	while also
	\begin{equation*}
		\prod_{n=k_1}^{N-1} a_n < \prod_{n=k_1+1}^{N-1} z^{(M+2+\delta)^{n}}
		= z^{\sum_{n=k_1+1}^{N-1}(M+2+\delta)^{n}}.
	\end{equation*}
	From this and equation \eqref{eq:boundM+1+delta}, we obtain
	\begin{equation*}
		\prod_{n=k_1}^{N-1} a_n < z^{\frac{(M+2+\delta)^{N}}{M+1+\delta}},
	\end{equation*}
	while inequality \eqref{eq:case1k1} together with the facts that $N>k_1$ and $a_n$ is non-decreasing yields
	\begin{equation*}
		\prod_{n=1}^{k_1-1} a_n \le a_{k_1}^{k_1-1} < z^{N^2}.
	\end{equation*}
	Thus,
	\begin{equation}\label{eq:case1prodBound}
		\prod_{n=1}^{N-1} a_n^{M+1} < z^{(M+1)N^2 + (M+1) \frac{(M+2+\delta)^{N}}{M+1+\delta}}.
	\end{equation}

	Having a bound for the product in $Z_N$, we move on to bounding the infinite series.
	Let $0<\Gamma<1$ be given as in Lemma \ref{lem:sumBound an>2n sometimes}.
	Let $\zeta\in\big(\frac{M+1}{M+1+\delta}, 1\big)$ be a fixed number that does not depend on $z$.
	Since $k_1$ is the largest number less than $k_2$ satisfying \eqref{eq:case1k1}, we have $a_{n} > z^{n} > 2^{n}$ for $N\le n\le k_2$.
	By Lemma \ref{lem:sumBound an>2n sometimes}, this means that when $z$ (and thereby $N$) is sufficiently large, then 
	\begin{equation*}
		\sum_{n=N}^{k_2} \frac{b_n}{a_n}\le \frac{2^{(\log_2 a_N)^\Gamma}}{a_N} \le \frac{1}{a_N^{\zeta}}.
	\end{equation*}
	This and inequality \eqref{eq:case1N} imply
	\begin{equation*}
		\sum_{n=N}^{k_2} \frac{b_n}{a_n} \le z^{-(M+2+\delta)^N\zeta}.
	\end{equation*}
	From Lemma \ref{lem:sumBound general}, inequality \eqref{eq:case1k2}, and the fact that $a_n$ is non-decreasing, it follows for all sufficiently large $z$ that
	\begin{align*}
		\sum_{n=k_2+1}^{\infty}\frac{b_n}{a_n}
		\le a_{k_2+1}^{-\gamma} \le a_{k_2}^{-\gamma} \le z^{-(M+2+\delta)^{k_2}},
	\end{align*}
	and so we have
	\begin{equation}\label{eq:case1sumBound}
		\sum_{n=N}^{\infty}\frac{b_n}{a_n} \le z^{-(M+2+\delta)^N\zeta} + z^{-(M+2+\delta)^{k_2}}
		\le \frac{2}{z^{(M+2+\delta)^N\zeta}}.
	\end{equation}
	
	Write $\zeta' = \zeta- \frac{M+1}{M+1+\delta}$ and note that $\zeta'>0$.
	By inequalities \eqref{eq:case1prodBound}, \eqref{eq:case1sumBound}, and \eqref{eq:case1-2logBound}, we find that for sufficiently large $z$ (and thus $N$),
	\begin{align*}
		Z_N&=2^{N^2 (\log_2 a_{N-1})^c}\Bigg(\prod_{n=1}^{N-1} a_n^M\Bigg) \sum_{n=N}^{k_2}\frac{b_n}{a_n}
		\\&
		< z^{N^2 \left(M+2+\delta\right)^{Nc} + (M+1)N^2 + (M+2+\delta)^{N} \left( \frac{M+1}{M+1+\delta} - \zeta	\right)}
		\\&
		= z^{N^2 (M+2+\delta)^{Nc} + MN^2 - \zeta' \left(M+2+\delta\right)^{N}}
		<
		z^{- \frac{\zeta'}{2} \left(M+2+\delta\right)^{N}}.
	\end{align*}
	Hence, $Z_N$ tends to $0$ when $z$ grows toward infinity, and this case is complete.
	
	\textbf{Case 2 (equation \eqref{eq:diverges fast} does not hold for any $\delta>0$).}
	This case will be split into further 2 subcases, depending on whether $a_n<2^n$ infinitely often.
	Before we do that, we make an observation to be used for both cases.
	
	Given any fixed number $\Gamma$, set $\Gamma_0 = \max\{c,\Gamma\}$ and $\tilde{\Gamma}=(1+2\Gamma_0)/(2+\Gamma_0)$.
	Note that $0<\tilde{\Gamma}<1$.
	By the assumption that equation \eqref{eq:diverges fast} holds for no $\delta>0$, we may pick a small number $\delta_\Gamma>0$ that does not depend on $n$, such that
	\begin{equation*}
		(M+2 + \delta_\Gamma)^{(2+\Gamma_0)/3} < M+2
	\end{equation*}
	and $a_n < 2^{\left(M+2+\delta_\Gamma\right)^{n}}$ for all sufficiently large $n$.
	Due to this and the fact that $a_n$ is non-decreasing,
	\begin{align}\nonumber
		n^2 (\log_2 a_{n-1})^{c} + (\log_2 a_n)^{\Gamma}
		&< 2n^2  (\log_2 a_n)^{\tilde \Gamma}
		<2 n^2 \big(M+2+\delta_\Gamma\big)^{n\Gamma_0}
		\\&	\nonumber
		< 2n^2 (M+2)^{n\frac{3\Gamma_0}{2+\Gamma_0}}
		< (M+2)^{n\frac{1+2\Gamma_0}{2+\Gamma_0}} 
		\\&	\label{eq:case2bound2log2}
		= (M+2)^{\tilde{\Gamma}n},
	\end{align}
	for all sufficiently large $n$.
	Recalling $(1+n^{-2})=2^{\log_2(1+n^{-2})}$ and using the Taylor expansion of $\log_2 x$ around 1, we have that $(1+n^{-2})>2^{n^{-3}}$ when $n$ is large.
	From this and inequality \eqref{eq:case2bound2log2}, we obtain for all large enough $n$ that
	\begin{align}\label{eq:case2ZN}
		\frac{2^{n^2 (\log_2 a_n)^{c} + (\log_2 a_n)^{\Gamma}}}{(1+n^{-2})^{(M+2)^n}}
		< \frac{2^{(M+2)^{\tilde{\Gamma}n}}}{2^{n^{-3} (M+2)^n}}.
	\end{align}

%
	
	\textbf{Case 2a ($a_n< 2^n$ for at most finitely many $n$).}
	By Lemma \ref{lem:sumBound an>2n}, we may pick $0<\Gamma<1$ such that for all sufficiently large $N$,
	\begin{equation*}
		\sum_{n=N}^{\infty}\frac{b_n}{a_n} \le  \frac{2^{(\log_2 a_N)^\Gamma}}{a_N}.
	\end{equation*}
	This and Corollary \ref{cor:>1+1/N2} imply that for infinitely many $N$,
	\begin{align*}\nonumber
		Z_N  &< 2^{N^2 (\log_2 a_{N-1})^c} \frac{a_N}{(1+N^{-2})^{\left(M+2\right)^N}} \frac{2^{(\log_2 a_N)^\Gamma}}{a_N}
		\\&
		= \frac{2^{N^2 (\log_2 a_{N-1})^{c} + (\log_2 a_{N-1})^{\Gamma}} }{2^{\left(M+2\right)^N \log_2 (1+N^{-2})}}.
	\end{align*}
	From this and inequality \eqref{eq:case2ZN}, we then have
	\begin{align*}
		Z_N&
		< \frac{2^{N^2\left(M+2\right)^{\tilde{\Gamma}N} } }{2^{N^{-3} \left(M+2\right)^N}}
	\end{align*}
	for infinitely many $N$.
	Since $\tilde{\Gamma}<1$, the right-hand-side tends to 0 as $N$ grows towards infinity, and we are done.
	
	\textbf{Case 2b ($a_n < 2^n$ infinitely often).}
	Let $C>0$ be sufficiently large.
	Pick $k_1,k_2,N\in\bN$ that depend on $C$ as follows.
	Let $k_2$ be the smallest integer such that
	\begin{align}\label{eq:case2bk2}
		a_{k_2}^{\left(M+2 \right)^{-k_2}} > C,
	\end{align}
	and let $k_1$ be the largest integer such that $k_1<k_2$ and
	\begin{equation}\label{eq:case2bk1}
		a_{k_1} < 2^{k_1}.
	\end{equation}
	Since $k_2\to\infty$ as $C\to\infty$, the assumption that $a_n<2^n$ infinitely often implies that also $k_1\to\infty$ as $C\to\infty$.
	Using Corollary \ref{cor:>1+1/N2}, pick $N>k_1$ to be the smallest integer such that
	\begin{equation}\label{eq:case2bN}
		a_N > \bigg(1 + \frac{1}{N^2}\bigg)^{ (M+2)^N}\Big( \max_{k_1\le n< N} a_n^{(M+2)^{-n}} \Big)^{(M+2)^N}.
	\end{equation}
	Consequently, we get by induction that if $k_1<n<N$, then
	\begin{align*}
		a_n^{\left(M+2\right)^{-n}} &\le \bigg(1+\frac{1}{n^2}\bigg)\max_{k_1\le i< n} a_i^{\left(M+2\right)^{-i}} 
		\\&
		\le \cdots\le a_{k_1}^{\left(M+2\right)^{-k_1}} \prod_{m=k_1+1}^{n} \bigg(1+\frac{1}{m^2}\bigg).
	\end{align*}
	Therefore, each $n$ with $k_1\le n<N$ must satisfy
	\begin{equation}\label{eq:case2banBound}
		a_n < 8^{(M+2)^n} = 2^{3 (M+2)^n},
	\end{equation}
	using that $\prod_{m=1}^{\infty} (1+m^{-2})<4$ and that $a_{k_1}^{(M+2)^{-k_1}}<2$ by inequality \eqref{eq:case2bk1}.
	Note that inequalities \eqref{eq:case2bk2} and \eqref{eq:case2banBound} ensure $N\le k_2$ when $C$ is large.
	
	Using inequality \eqref{eq:case2bk1} along with the facts that $a_n$ is non-decreasing and $N>k_1$, we find
	\begin{equation}\label{eq:case2bProdK1}
		\prod_{n=1}^{k_1-1} a_n \le a_{k_1-1}^{k_1-1} < a_{k_1}^{k_1} < 2^{k_1^2} < 2^{N^2}.
	\end{equation}
	From inequalities \eqref{eq:case2banBound} and \eqref{eq:case2bProdK1}, we get
	\begin{equation*}
		\prod_{n=1}^{N-1} a_n^{M+1} = \Bigg(\prod_{n=1}^{k_1-1} a_n^{M+1}\Bigg) \prod_{n=k_1}^{N-1} a_n^{M+1} 
		<2^{N^2(M+1)+ 3 \sum_{n=k_1}^{N-1} (M+2)^n}.
	\end{equation*}
	Due to this and equation \eqref{eq:boundM+1+delta}, all large enough $C$ (and thus $N$) must satisfy
	\begin{align}\nonumber
		\prod_{n=1}^{N-1} a_n^{M+1} 
		&<2^{N^2(M+1)+ 3 \sum_{n=k_1}^{N-1} (M+2)^n}
		< 2^{N^2(M+1) + 3(M+2)^N} 
		\\& \label{eq:case2bprodBound1}
		<2^{4(M+2)^N}.
	\end{align}
	Meanwhile, inequality \eqref{eq:case2bN} allows us to apply Corollary \ref{cor:>1+1/N2} and find
	\begin{equation*}
		\prod_{n=k_1}^{N-1} a_n^{M+1} 
		\le  \frac{a_N}{(1+N^{-2})^{\left(M+2\right)^N}},
	\end{equation*}
	which together with inequality \eqref{eq:case2bProdK1} yields
	\begin{align}\nonumber
		\prod_{n=1}^{N-1} a_n^{M+1} 
		&= \Bigg(\prod_{n=1}^{k_1-1} a_n^{M+1}\Bigg) \prod_{n=k_1}^{N-1} a_n^{M+1} 
		\\&\label{eq:case2bprodBound2}
		< 2^{N^2(M+1)} \frac{a_N}{(1+N^{-2})^{\left(M+2\right)^N}}.
	\end{align}
	
	Having two upper bounds of $\prod_{n=1}^{N-1} a_n^{M+1}$, we move on to also bounding $\sum_{n=N}^{\infty}b_n/a_n$.
	Pick $\Gamma\in(0,1)$ as in Lemma \ref{lem:sumBound an>2n sometimes}.
	Since $k_1< N\le k_2$, $k_1$ is the greatest integer less than $k_2$ that satisfies inequality \eqref{eq:case2bk1}, and $a_{k_2}> 2^{k_2}$ due to inequality \eqref{eq:case2bk2}, we obtain from Lemma \ref{lem:sumBound an>2n sometimes} that
	\begin{equation}\label{eq:case2bSum partial}
		\sum_{n=N}^{k_2}\frac{b_n}{a_n} \le \frac{2^{(\log_2 a_N)^\Gamma}}{a_N}.
	\end{equation}
	Similarly, pick $\gamma\in(0,1)$ as in Lemma \ref{lem:sumBound general}.
	Then
	\begin{equation*}
		\sum_{n=k_2+1}^{\infty}\frac{b_n}{a_n} 
		\le \frac{1}{a_{k_2+1}^{\gamma}} \le \frac{1}{a_{k_2}^{\gamma}},
	\end{equation*}
	since $a_n$ is non-decreasing. 
	Estimating this further by applying inequality \eqref{eq:case2bk2} together with the fact that $k_2\ge N$, we get
	\begin{align*}
		\sum_{n=k_2+1}^{\infty}\frac{b_n}{a_n} 
		\le \frac{1}{a_{k_2}^{\gamma}}
		< \frac{1}{C^{\gamma (M+2)^{k_2}}}
		\le \frac{1}{C^{\gamma (M+2)^N}}.
	\end{align*}
	This and inequality \eqref{eq:case2bSum partial} imply
	\begin{equation}\label{eq:case2bSum}
		\sum_{n=N}^{\infty}\frac{b_n}{a_n}
		= \sum_{n=N}^{k_2} \frac{b_n}{a_n} + \sum_{n=k_2+1}^{\infty} \frac{b_n}{a_n}
		\le \frac{2^{(\log_2 a_N)^\Gamma}}{a_N} + \frac{1}{C^{\gamma (M+2)^N}}.
	\end{equation}
	
	From inequalities \eqref{eq:case2bprodBound1}, \eqref{eq:case2bprodBound2}, and \eqref{eq:case2bSum}, it follows that
	\begin{align}\nonumber
		Z_N =\,& 2^{N^2 (\log_2 a_{N-1})^c}\Bigg(\prod_{n=1}^{N-1} a_n\Bigg)^{M+1} \sum_{n=N}^{k_2}\frac{b_n}{a_n}
		\\\nonumber
		<\,& 
		2^{N^2 (\log_2 a_{N-1})^c} 
		\min\bigg\{
			2^{4(M+2)^N}, 
			\frac{2^{N^2(M+1)}a_N}{(1+N^{-2})^{\left(M+2\right)^N}}
		\bigg\}
		\\& \nonumber \cdot
		\bigg(\frac{2^{(\log_2 a_N)^\Gamma}}{a_N} + \frac{1}{C^{\gamma (M+2)^N}}\bigg)
		\\\label{eq:case2bZN}
		\le\,& \frac{ 2^{N^2 (\log_2 a_{N-1})^c + N^2(M+1) + (\log_2 a_N)^{\Gamma} }}{(1+N^{-2})^{\left(M+2\right)^N}}
		+ 
		\frac{2^{N^2 (\log_2 a_{N-1})^c + 4 \left(M+2\right)^N} }{C^{\gamma (M+2)^N}}.
	\end{align}
	By inequality \eqref{eq:case2ZN},
	\begin{equation*}
		\frac{ 2^{N^2 (\log_2 a_{N-1})^c + N^2(M+1) + (\log_2 a_N)^{\Gamma} }}{(1+N^{-2})^{\left(M+2\right)^N}}
		<
		\frac{2^{N^2(M+1)} 2^{N^2 (M+2)^{\tilde{\Gamma}N}}}{2^{N^{-3} (M+2)^N}},
	\end{equation*}
	which tends to $0$ as $C$ grows large, due to the facts that $N>k_1$, $k_1\to\infty$ as $C\to\infty$, and $\tilde{\Gamma}<1$.
	We are thus left to show that also the second term of \eqref{eq:case2bZN} goes to 0 when $C$ (and thus $N$) grows toward infinity. It follows from inequality \eqref{eq:case2banBound} that $\log_2 a_{N-1} < 3 (M+2)^N$, and so we estimate
	\begin{equation*}
		\frac{2^{N^2 (\log_2 a_{N-1})^c + 4 \left(M+2\right)^N} }{C^{\gamma (M+2)^N}}
		< \frac{2^{N^2 3^c \left(M+2\right)^{cN} + 4 \left(M+2\right)^N} }{C^{\gamma (M+2)^N}}.
	\end{equation*}
	Since $\gamma>0$ and $c<1$, this clearly tends to 0 as $C$ tends to infinity, and the proof is complete.
\end{proof}

Lemma \ref{lemma:hns} is to be used in connection with the below elementary result.
\begin{Lemma}\label{lemma:distance in polynomial}
	Let $R\ge 1$, let $P\in \mathbb{Z}[x_1,\ldots, x_K]$ be a polynomial, and let 
	$\alpha_1,\ldots,\alpha_K, \beta_1,\ldots,\beta_K\in\bC$ with $|\alpha_k|, |\beta_k|\le R$ for each $k$.
	Then there is a constant $C>0$, depending only on $P$ and $R$, such that
	\begin{equation*}
		|P(\alpha_1,\ldots, \alpha_K) - P(\beta_1,\ldots, \beta_K)|\le C \max_{1\le k\le K} |\alpha_k - \beta_k|
	\end{equation*}
\end{Lemma}
\begin{proof}
	
	For $i_k\in\bN$, note that
	\begin{align}\nonumber
		\prod_{k=1}^{K} \alpha_k^{i_k}  - \prod_{k=1}^{K} \beta_k^{i_k}	&
		= \sum_{k=1}^{K} \Bigg(\prod_{j=1}^{k-1} \alpha_l^{i_l}\Bigg) (\alpha_k^{i_k} - \beta_k^{i_k}) \prod_{j=k+1}^{K} \beta_l^{i_l}
		\\&\label{eq:monomial diff}
		= \sum_{k=1}^{K} \Bigg((\alpha_k - \beta_k)\sum_{j=1}^{i_k} \alpha_k^{j-1} \beta_k^{k-j} \Bigg)\Bigg(\prod_{j=1}^{k-1} \alpha_l^{i_l}\Bigg) \prod_{j=k+1}^{K} \beta_l^{i_l}.
	\end{align}
	Write
	$P(x_1,\ldots,x_K) = \sum_{i_1,\ldots, i_k} c_{i_1,\ldots, i_K} x_1^{i_1}\cdots x_K^{i_K}$, $d=\deg P$, $\alpha=(\alpha_1,\ldots,\alpha_K)$, and $\beta=(\beta_1,\ldots,\beta_K)$.
	Then the triangle inequality, equation \eqref{eq:monomial diff}, and the facts that $\deg P = d$ and $|\alpha_k|, |\beta_k|\le R$ let us conclude
	\begin{align*}
		|P(\alpha_1,\ldots, \alpha_K) - P(\beta_1,\ldots&, \beta_K)| = \Bigg|\sum_{i_1,\ldots, i_k} c_{i_1,\ldots,i_K} \Bigg( \prod_{k=1}^{K} \alpha_k^{i_k}  - \prod_{k=1}^{K} \beta_k^{i_k} \Bigg)\Bigg|
		\\
		\le&\, \sum_{i_1,\ldots, i_k} |c_{i_1,\ldots,i_K}| \sum_{k=1}^{K} |\alpha_k - \beta_k| d R^{d-1}
		\\
		\le&\, \Bigg(	dK R^{d-1}\sum_{i_1,\ldots, i_k} |c_{i_1,\ldots,i_K}|	\Bigg) \max_{1\le k\le K}|\alpha_k - \beta_k|.
		\qedhere
	\end{align*}
\end{proof}

In order to handle the case of $K=2$ and $d>3$ in Theorems \ref{thm:general} and \ref{thm:general infinite}, we will need a few results from algebraic geometry.
Let $\bP^K$ denote the $K$-dimensional complex projective space, i.e., $\cP={(\bC^{K+1}\setminus 0)}\big/{\sim}$ where $\sim$ denotes the equivalence relation that $x\sim y$ if $x = \alpha y$ for some $\alpha\in\bC\setminus 0$.
For $i=1,\ldots, n$, let  $P_i\in\bC[x_1,\ldots, x_K]$ be a polynomial, and let $\tilde{P}_i\in\bC[x_1,\ldots,x_{K+1}]$ denote the unique homogeneous polynomial of minimal degree such that $P_i(x_1,\ldots,x_K)=\tilde{P}_i(x_1\ldots,x_K,1)$.
We then say that the set 
\begin{equation*}
	V(P_1,\ldots,P_n) = \{[x]\in\bP^K : \tilde{P}_1(x)=\cdots=\tilde{P}_n(x)=0\}
\end{equation*}
is a \textit{projective variety}. If the coefficients of $P_1,\ldots, P_n$ are all rational numbers, we say that 
$V(P_1,\ldots,P_n)$
is a projective
variety over $\bQ$.

Let $V_1$ and $V_2$ be two non-empty irreducible projective varieties that are proper subvarieties of $\bP^K$.
We then say that $f:V_1\dashrightarrow V_2$ is a \textit{rational map} if there is a proper subvariety $B_1\subsetneq V_1$ (which may be reducible) such that $f$ restricted to $V_1\setminus B_1$ is a well-defined rational function, i.e., each coordinate map is given as the quotient of two polynomials.
If there is another rational map $g:V_2\dashrightarrow V_1$ so that $g\circ f$ coincides with the identity where it is defined, then we say that $f$ is a \textit{birational map} and that $V_1$ and $V_2$ are \textit{birationally equivalent}.
We furthermore say that $f$ is a rational (resp. birational) map over $\bQ$ if the implied rational function can be chosen so that the coefficients of its coordinate maps are contained in $\bQ$, and we say that $V_1$ and $V_2$ are birationally $\bQ$-equivalent if there is a birational map $V_1\dashrightarrow V_2$ over $\bQ$.
We will need the below three theorems, of which the first is known as Fatings's Theorem, and the second is a consequence of the degree--genus formula.
Genus here refers to the geometric genus.
All three theorems can be found in \cite{HinSil}.
\begin{Theorem}[Faltings]
	Let $A$ be a non-singular irreducible projective variety over $\bQ$ of genus $g>1$ and dimension 1.
	Then $A$ has only finitely many rational points.
\end{Theorem}
\begin{Theorem}
	Let $V\subseteq\bP^2$ be a non-singular irreducible projective variety of degree $d$ and dimension 1.
	Then the genus of $V$ equals $(d-1)(d-2)/2$.
\end{Theorem}
\begin{Theorem} \label{thm:A414}
	Let $V\subseteq \mathbb{P}^2$ be an irreducible singular projective variety over $\bQ$ of dimension 1.
	Then $V$ is birationally $\bQ$-equivalent to a smooth irreducible projective variety over $\bQ$.
\end{Theorem}

Finally, we will also need the below simple result.
\begin{Lemma}
	Let $V_1, V_2$ be birationally equivalent irreducible projective varieties of dimension 1.
	Then the implied rational map $f:V_1\dashrightarrow V_2$ is defined for all but finitely many elements of $V_1$.
\end{Lemma}
\begin{proof}
	Because $V_1$ is irreducible, any non-empty proper subvariety $B$ of $V_1$ must be of a lower dimension than $V_1$.
	Since $V_1$ is of dimension 1, this makes $B$ finite.
\end{proof}

\section{Proof of Theorems \ref{thm:general} and \ref{thm:general infinite}}
For this section, let $d$ and $K$ be positive integers, and let $\{a_{k,n}\}_{n=1}^\infty$ and $\{b_{k,n}\}_{n=1}^\infty$ be sequences of non-zero integers such that each of the sequences $\{\alpha_{1,N}\}_{N=1}^\infty,\ldots, \{\alpha_{K,N}\}_{N=1}^\infty$ defined by
\begin{equation*}
	\alpha_{k,N} = \sum_{n=1}^{N} \frac{b_{k,n}}{a_{k,n}}
\end{equation*}
converge, and we write $\alpha_k = \lim_{N\to\infty}\alpha_{k,N}$ for the corresponding limit.

Theorems \ref{thm:general} and \ref{thm:general infinite} have almost identical proofs.
For that reason, we will prove them simultaneously.
Roughly speaking, their proofs can be divided into an analytical part, which is covered by the below lemma, and an algebraic part, which is covered by the subsequent two lemmas.
\begin{Lemma}\label{lemma:analytic}
	Let $\varepsilon$ and $\kappa$ be positive real numbers with $\kappa <1$.
	Suppose for all $k=1,\dots ,K$ and $n\in\bN$ that equations \eqref{hkl4}, \eqref{hkl5}, \eqref{hkl6}, and \eqref{hkl7} are satisfied.
	Let $P\in\mathbb{Z}[x_1,\ldots,x_K]$ be a polynomial of degree at most $d$.
	If $P(\alpha_{1,N},\ldots\alpha_{K,N}) \ne 0$ for all sufficiently large $N$, then $P(\alpha_{1},\ldots\alpha_{K}) \ne 0$.
\end{Lemma}

\begin{proof}
Let $N$ be sufficiently large.
Then $P(\alpha_{1,N},\ldots,\alpha_{K,N})\ne 0$.
Therefore, since $\deg P\le d$ and $\alpha_{k,N}\prod_{n=1}^{N}|a_{k,n}|$ must be integral, we get
\begin{equation*}
	| P(\alpha_{1,N},\ldots, \alpha_{K,N})| 
	\ge \Bigg(
	\prod_{n=1}^{N}\prod_{k=1}^{K}|a_{k,n}|\Bigg)^{-d}
\end{equation*}
and so, by inequality \eqref{hkl5},
\begin{equation}\label{eq:P(alphakN)bound}
	| P(\alpha_{1,N},\ldots, \alpha_{K,N})| \ge \prod_{n=1}^{N}\Big( a_n^K 2^{(K-1)(\log_2 a_n)^\kappa}  \Big)^{-d}.
\end{equation}
Meanwhile, Lemma \ref{lemma:distance in polynomial} implies that there is a $C>0$ depending only on $P$ and on $\sup_{k,N}|\alpha_{k,N}|$ such that
\begin{align*}
	|P(\alpha_1,\ldots,\alpha_K) - P(\alpha_{1,N},\ldots,\alpha_{K,N})|
	&\le C\max_{1\le k\le K} |\alpha_k -\alpha_{k,N}|
	\\&
	= C\max_{1\le k\le K}\Bigg|	\sum_{n=N+1}^{\infty} \frac{b_{k,n}}{a_{k,n} }	\Bigg|.
\end{align*}
Therefore, due to the triangle inequality followed by inequalities \eqref{hkl5} and \eqref{hkl6},
\begin{align*}
	|P(\alpha_1,\ldots,\alpha_K) - P(\alpha_{1,N},\ldots,\alpha_{K,N})|
	&\le C\max_{1\le k\le K}	\sum_{n=N+1}^{\infty} \frac{|b_{k,n}|}{|a_{k,n}| }	
	\\
	\le C \sum_{n=N+1}^{\infty} \frac{2^{2(\log_2 a_n)^\kappa}}{a_n } 
	&= C\sum_{n=N+1}^{\infty} \frac{4^{(\log_2 a_n)^\kappa}}{a_n }.
\end{align*}
This, the triangle inequality, and estimate \eqref{eq:P(alphakN)bound} imply that
\begin{align*}
	|P(\alpha_1,\ldots,\alpha_K)| \ge &|P(\alpha_{1,N},\ldots,\alpha_{K,N})| - |P(\alpha_1,\ldots,\alpha_K) - P(\alpha_{1,N},\ldots,\alpha_{K,N})| 
	\\
	\ge&  \Bigg( \prod_{n=1}^{N}a_n^K 2^{(K-1)(\log_2 a_n)^\kappa}  \Bigg)^{-d}
	-  C\sum_{n=N+1}^{\infty} \frac{4^{(\log_2 a_n)^\kappa}}{a_n}.
\end{align*}
Hence, to prove the lemma, it is enough to show that there are infinitely many $N$ such that
\begin{equation*}
	\Bigg( \prod_{n=1}^{N}a_n^K 2^{(K-1)(\log_2 a_n)^\kappa}  \Bigg)^{-d} 
	> C\sum_{n=N+1}^{\infty} \frac{4^{(\log_2 a_n)^\kappa}}{a_n}
\end{equation*}
or, equivalently,
\begin{equation}\label{eq:analytical}
	\Bigg( \prod_{n=1}^{N}a_n^K 2^{(K-1)(\log_2 a_n )^\kappa}  \Bigg)^{d} \Bigg(\sum_{n=N+1}^{\infty} \frac{4^{(\log_2 a_n)^\kappa}}{a_n}\Bigg) < C^{-1}.
\end{equation}
Note that
\begin{align}\nonumber
	\Bigg( \prod_{n=1}^{N}&a_n^K 2^{(K-1)(\log_2 a_n )^\kappa}  \Bigg)^{d} \sum_{n=N+1}^{\infty} \frac{4^{(\log_2 a_n)^\kappa}}{a_n}
	\\&\label{eq:analyticZn}
	\le 2^{N^2 (\log_2a_N)^\kappa} \Bigg(\prod_{n=1}^{N} a_n\Bigg)^{dK}  \sum_{n=N+1}^{\infty} \frac{\left\lceil 4^{(\log_2 a_n)^{\kappa}}\right\rceil}{a_n},
\end{align}
for all large enough values of $N$.
Taking $\kappa'\in(\kappa,1)$, we have
$\left\lceil 4^{(\log_2 a_n)^{\kappa}}\right\rceil\le 2^{(\log_2 a_n)^{\kappa'}}$ when $N$ is sufficiently large.
Using this together with assumptions \eqref{hkl4} and \eqref{hkl7}, we may apply Lemma \ref{lemma:hns} with $M=dK-1$.
Therefore, we can pick infinitely many $N$ such that the right-hand-side of \eqref{eq:analyticZn} is smaller than $C^{-1}$, and so inequality \eqref{eq:analytical} follows, and the proof is complete.
\end{proof}

For the remaining two lemmas, let $\cP$ be a set of either a single prime number or infinitely many prime numbers, and assume that $p\nmid\gcd(b_{k,n}, a_{k,n})$ for all $k=1,\ldots, K$ and $n\in\bN$.
Recall the elementary facts that
\begin{equation}\label{eq:nupRules}
	\nu_p(ab) = \nu_p(a)+\nu_p(b),
	\qquad
	\nu_p(a+b) \begin{cases}
		= \nu_p(a),	&\text{if } \nu_p(a)<\nu_p(b),	\\
		\ge \nu_p(a),	&\text{if } \nu_p(a)=\nu_p(b).
	\end{cases} 
\end{equation}
From the first of these two facts, $\nu_p$ extends to a function $\bQ\to\bZ$ by $\nu_p(0)=\infty$ and $\nu_p(q) = \nu_p(a) - \nu_p(b)$ when $q=a/b$; note that $\nu_p(q)$ does not depend on the choice of representative $(a,b)$.

In the proof of the below lemma, we will order tuples of indices colexicographically, i.e., we will say that
\begin{equation*}
	(i_1,\ldots,i_{n+1}) < (j_1,\ldots,j_{n+1})
\end{equation*}
if $i_{n+1} < j_{n+1}$, or if both $i_{n+1} = j_{n+1}$ and $(i_1,\ldots, i_n) < (j_1,\ldots, j_n)$.

\begin{Lemma}\label{lemma:algebraic general}
	Suppose for each fixed $k=1,\ldots,K$ and $p\in\cP$ that equation \eqref{hkl1} holds for all $N\in\bN$ and that inequality \eqref{hkl3} holds for all sufficiently large $N$.
	If $\cP=\{p\}$, assume additionally that equation \eqref{hkl2} holds for each $k=1,\ldots,K$.
	Let $P\in\mathbb{Z}[x_1,\ldots,x_K]$ be a non-zero polynomial of degree at most $d$.
	Then $P(\alpha_{1,N},\ldots, \alpha_{K,N})\ne 0$ for all sufficiently large $N$.
\end{Lemma}

\begin{proof}
	Write
	\begin{align*}
		P(x_1,\ldots,x_K) = \sum_{i_1,\ldots, i_K} c_{i_1,\ldots,i_K} x_1^{i_1} \cdots x_K^{i_K},
	\end{align*}
	and let $c_{j_1,\ldots,j_K} x_1^{j_1} \cdots x_K^{j_K}$ be its leading term in colexicographic order.
	We will now prove that for a suitably chosen $p\in\cP$ and each sufficiently large $N$, the term $c_{j_1,\ldots,j_K} \alpha_{1,N}^{j_1} \cdots \alpha_{K,N}^{j_K}$ has strictly lower $p$-adic valuation than any of the other terms of $P(\alpha_{1,N},\ldots,\alpha_{K,N})$.
	By \eqref{eq:nupRules}, this will then imply the lemma.	
	To simplify notation, write
	\begin{equation}\label{eq:Ci}
		C_{i_1,\ldots, i_K} = \frac{c_{i_1,\ldots,i_K}}{ c_{j_1,\ldots,j_K} } 	\prod_{k=1}^{K}\alpha_{k,N}^{i_k-j_k}.
	\end{equation}
%
	If $|\cP|=1$, let $p$ be the unique prime in $\cP$.
	Otherwise, $|\cP|$ is infinite, and we instead pick $p$ such that 
	\begin{equation}\label{eq:nupci}
		\nu_p(c_{j_1,\ldots, j_K})= 0.
	\end{equation}
	Since clearly $C_{j_1,\ldots, j_K}=1$ by equation \eqref{eq:Ci}, the lemma follows if we can prove for all $(i_1,\ldots,i_K)\ne(j_1,\ldots, j_K)$ with $c_{i_1,\ldots, i_K}\ne 0$ that
	\begin{equation}\label{eq:nupCi}
		\nu_p(C_{i_1,\ldots, i_K})>0
	\end{equation}
	when $N$ is sufficiently large.
	
	Suppose $(i_1,\ldots,i_K)\ne(j_1,\ldots, j_K)$ with $c_{i_1,\ldots, i_K}\ne 0$, and let $m$ be the largest index such that $i_k\ne j_k$.
	Note that this implies $i_m<j_m$ since $c_{j_1,\ldots,j_K} x^{j_1}\cdots x^{j_K}$ is the leading term of $P$.
	When $N$ is sufficiently large, we have from inequality \eqref{hkl3} that $\max_{1\le n\le N}\nu_p(a_{k,n})>0$ for each $k=1,\ldots, K$, which together with \eqref{hkl1}, \eqref{eq:nupRules}, and the fact that $p\nmid\gcd(a_{k,n},b_{k,n})$ implies that $\nu_p(\alpha_{k,N}) = - \max_{1\le n\le N} \nu_p(a_{k,n})$.
	Using this and equation \eqref{eq:Ci}, we find that when $N$ is sufficiently large,
	\begin{align*}
		\nu_p(C_{i_1,\ldots, i_K}) &= \nu_p(c_{i_1,\ldots, i_K}) - \nu_p(c_{j_1,\ldots,j_K}) + \sum_{k=1}^{K} (i_k-j_k)\nu_p(\alpha_{k,N}) 
		\\&
		= \nu_p(c_{i_1,\ldots, i_K}) -\nu_p(c_{j_1,\ldots,j_K}) + \sum_{k=1}^{m} (j_k-i_k)\max_{1\le n\le N}\nu_p(a_{k,n}).
	\end{align*}
	Hence, using the facts that $i_m<j_m$, $c_{i_1,\ldots,i_K}\in\bZ$, and $a_{k,n}\in\bZ$,
	\begin{equation*}
		\nu_p(C_{i_1,\ldots, i_K}) \ge -\nu_p(c_{j_1,\ldots,j_K}) 
		+ \max_{1\le n\le N}\nu_p(a_{m,n})
		- \sum_{k=1}^{m-1} i_k\max_{1\le n\le N}\nu_p(a_{k,n}).
	\end{equation*}
	This, the fact that $\deg P\le d$, and inequality \eqref{hkl3} yield
	\begin{equation}\label{eq:nupCiBound}
		\nu_p(C_{i_1,\ldots, i_K})\ge -\nu_p(c_{j_1,\ldots,j_K}) + \max_{1\le n\le N} \nu_p(a_{m,n}) - d\max_{1\le n\le N}\nu_p(a_{m-1,n}).
	\end{equation}
	If $|\cP|=1$, then inequality \eqref{eq:nupCi} follows from inequality \eqref{eq:nupCiBound} and equation \eqref{hkl2} when $N$ is sufficiently large.
	Similarly, if $|\cP|=\infty$, then \eqref{eq:nupCi} is verified by applying inequality \eqref{eq:nupci} and equation \eqref{hkl3} to inequality \eqref{eq:nupCiBound} when $N$ is sufficiently large.
	This completes the proof.
\end{proof}
The final lemma of this paper provides an improvement to Lemma \ref{lemma:algebraic general} when $K=2$ and the degree of $P$ is greater than 3. 
\begin{Lemma}\label{lemma:algebraic K=2}
	Suppose the assumptions of Lemma \ref{lemma:algebraic general} are satisfied for $d=3$ and $K=2$. 
	Let $P\in\mathbb{Z}[x_1,x_2]$ be a non-zero polynomial of arbitrary degree such that $P$ defines a curve of genus $\ge 2$.
	Then $P(\alpha_{1,N},\alpha_{2,N})\ne 0$ for all sufficiently large $N$.
\end{Lemma}

As before, we remark that choosing a different value $d'$ in place of $3$ will rule out polynomials of degree $\le d'$. This will make the genus requirement less restrictive. Indeed, there are fewer polynomials resulting in low genus if the degree becomes higher.

\begin{proof}[Proof of Theorem \ref{thm:general}]
	Since equation \eqref{hkl2} implies inequality \eqref{hkl3}, the theorem follows by combining Lemma \ref{lemma:analytic} with Lemmas \ref{lemma:algebraic general} and \ref{lemma:algebraic K=2}.
\end{proof}
\begin{proof}[Proof of Theorem \ref{thm:general infinite}]
	This follows by combining Lemma \ref{lemma:analytic} with Lemmas \ref{lemma:algebraic general} and \ref{lemma:algebraic K=2}.
\end{proof}
\begin{proof}[Proof of Lemma \ref{lemma:algebraic K=2}]
	Without loss of generality, $P$ is irreducible; if not, then we may write $P$ as the product of irreducible polynomials $P_1,\ldots, P_L\in\bZ[x_1,x_2]$, and we would then just have to prove the lemma for each of these $P_l$.	
	If $\deg P\le 3$, the statement follows from Lemma \ref{lemma:algebraic general}, so we may assume $\deg P>3$.
	
	Since $p\nmid\gcd(b_{k,n},a_{k,n})$ for all $p\in\cP$, equations \eqref{hkl1} and \eqref{eq:nupRules} together with either \eqref{hkl3} (if $|\cP|=\infty$) or \eqref{hkl2} (if $\cP=\{p\}$) implies that the minimal denominator of $\alpha_{k,N}$ tends to infinity as $N$ tends to infinity.
	Hence, $\alpha_{k,N}$ does not take the same value infinitely often.
	
		Introducing an extra variable $X_3$, let $\tilde{P}\in\mathbb{Z}[X_1,X_2,X_3]$ be the homogenised version of $P$, and consider the projective curve $\mathcal{C}_P = \{[x] \in \mathbb{P}^2 : \tilde{P}(x)=0\}$.
		Since $P(\alpha_{1,N},\alpha_{2,N})=\tilde{P}(\alpha_{1,N},\alpha_{2,N},1)$ and since $\alpha_{k,N}$ is always rational but does not take the same value infinitely often, it suffices to prove that $\mathcal{C}_P$ contains only finitely many rational points.
	
	
	Suppose first that $\mathcal{C}_P$ is non-singular. By the degree--genus formula \cite[Theorem A.4.2.6.]{HinSil}, the genus of $\mathcal{C}_P$ is at least $2$ since $P$ is irreducible with $\deg P>3$. By Faltings's Theorem, such a variety contains at most finitely many rational points, and the lemma is proven.
	
	Next, suppose that $\mathcal{C}_P$ is singular. In this case, we may normalise $V_P$, by appealing to e.g. \cite[Theorem A4.1.4.]{HinSil}, to obtain a smooth projective curve birationally equivalent to $\mathcal{C}_P$. As mentioned in the introduction, this may result in a curve of lower genus than predicted by the degree--genus formula. Ordinary singularities result in a decrease of $\frac{1}{2} r (r-1)$, where $r$ is the multiplicity of the singularity. Non-ordinary singularities decrease the genus by a quantity which must be calculated on a case-by-case basis, but such polynomials are very rare indeed. At any rate, the assumptions of the lemma is that the resulting curve has genus at least $2$, so by Falting's Theorem there are only finitely many rational points on $\mathcal{C}_P$. For the original polynomial (before normalisation), we observe that desingularising again would not increase the number of rational points, and we are done.
\end{proof}

\textbf{Acknowledgements.} 
We thank Professor Jan \v{S}ustek from the Department of Mathematics, University of Ostrava for valuable remarks on the paper.
S.~Kristensen and M.~L.~Laursen are supported by the Independent Research Fund Denmark (Grant ref. 1026-00081B) and the Aarhus University Research Research Foundation (Grant ref. AUFF-E-2021-9-20).

J. Han\v cl\\
Department of Mathematics, Faculty of Science, University of Ostrava, 30.~dubna~22, 701~03 Ostrava~1, Czech Republic.\\
e-mail: hancl@osu.cz

\smallskip

S. Kristensen\\
Department of Mathematics, Aarhus University, Ny Munkegade 118, 8000 Aarhus C, Denmark. \\
e-mail: sik@math.au.dk

\smallskip

M. L. Laursen\\
Department of Mathematics, Aarhus University, Ny Munkegade 118, 8000 Aarhus C, Denmark. \\
e-mail: mll@math.au.dk

\end{document}